\documentclass{amsart}

\usepackage[margin=1in]{geometry}

\usepackage[T1]{fontenc}
\usepackage{graphicx}
\usepackage{multicol,multirow}
\usepackage{amsmath,amssymb,amsfonts}
\usepackage{mathrsfs}
\usepackage{rotating}
\usepackage{appendix}
\usepackage{ifpdf}
\usepackage{newtxtext}
\usepackage{newtxmath}
\usepackage{textcomp}

\usepackage{xcolor}
\definecolor{weblmcolor}{RGB}{0,80,140}

\usepackage[colorlinks,allcolors=weblmcolor]{hyperref}

\usepackage{enumitem}
\usepackage{qtree}
\usepackage{tikz-cd}
\usepackage[makeroom]{cancel}
\usepackage{float}
\usepackage[capitalise]{cleveref}
\usepackage{thmtools}
\usepackage{etoolbox}

\usepackage[
  backend=biber,
  style=numeric,
  url=false,
  maxnames=10,
  minnames=5
]{biblatex}
\DeclareRobustCommand{\orcid}[1]{%
  \texorpdfstring{%
    \,\href{https://orcid.org/#1}%
      {\textsuperscript{\normalfont\scriptsize ORCID}}%
  }{}%
}

\newtheorem{manualtheoreminner}{Theorem}
\newenvironment{manualtheorem}[1]{%
  \ifstrempty{#1}
    {\renewcommand{\themanualtheoreminner}{}}
    {\renewcommand{\themanualtheoreminner}{#1}}%
  \manualtheoreminner
}{%
  \endmanualtheoreminner
}

\newtheorem{theorem}{Theorem}[section]
\newtheorem{proposition}[theorem]{Proposition}

\newtheorem{question}[theorem]{Question}

\theoremstyle{definition}
\newtheorem{definition}[theorem]{Definition}
\newtheorem{lemma}[theorem]{Lemma}

\numberwithin{equation}{section}

\newcommand{\define}[1]{\textbf{#1}}
\newcommand{\Etp}[2]{\exists\text{-tp}_{#1}(#2)}
\renewcommand{\phi}{\varphi}
\newcommand{\notE}{\cancel{E}}

\newcommand{\K}{\mathbb{K}}

\newcommand{\finitary}[2]{{{#1}^{fin}_{#2}}}
\newcommand{\A}{\mathcal{A}}
\newcommand{\B}{\mathcal{B}}
\newcommand{\D}{\mathcal{D}}
\newcommand{\J}{\mathcal{J}}
\newcommand{\G}{\mathcal{G}}
\renewcommand{\H}{\mathcal{H}}
\newcommand{\E}{\mathbin{E}}

\renewcommand{\path}[1]{\mathcal{P}_{#1}}

\newcommand{\cycle}[1]{\mathcal{C}_{#1}}
\newcommand{\free}[1]{\mathsf{Free}(#1)}
\newcommand{\finfree}[1]{\mathsf{Free}(#1)^{fin}}
\renewcommand{\S}{\mathcal{S}}
\newcommand{\T}{\mathcal{T}}

\newcommand{\C}{\mathbb{C}}
\renewcommand{\D}{\mathbb{D}}
\renewcommand{\paragraph}[1]{\smallskip\noindent{\textit{#1}}}
\begin{document}

\title[Dichotomy results for classes of countable graphs]
      {Dichotomy results for classes of countable graphs}

\author{Vittorio Cipriani\orcid{0000-0002-3847-2893}}
\address{%
  Institut f\"ur Diskrete Mathematik und Geometrie,
  Technische Universit\"at Wien,
  Wiedner Hauptstra{\ss}e 8--10,
  1040 Wien, Austria
}
\email{vittorio.cipriani@tuwien.ac.at}
\urladdr{https://vittoriocipriani.github.io/}

\author{Ekaterina Fokina\orcid{0000-0002-4598-458X}}
\address{%
  Institut f\"ur Diskrete Mathematik und Geometrie,
  Technische Universit\"at Wien,
  Wiedner Hauptstra{\ss}e 8--10,
  1040 Wien, Austria
}
\email{ekaterina.fokina@tuwien.ac.at}

\author{Matthew Harrison-Trainor}
\address{%
  Department of Mathematics, Statistics, and Computer Science,
  University of Illinois Chicago,
  Chicago, IL 60607, USA
}
\email{mht@uic.edu}

\author{Liling Ko\orcid{0000-0002-8888-5530}}
\address{%
  Institut f\"ur Diskrete Mathematik und Geometrie,
  Technische Universit\"at Wien,
  Wiedner Hauptstra{\ss}e 8--10,
  1040 Wien, Austria
}
\email{koliling@gmail.com}

\author{Dino Rossegger\orcid{0000-0003-3494-9049}}
\address{%
  Institut f\"ur Diskrete Mathematik und Geometrie,
  Technische Universit\"at Wien,
  Wiedner Hauptstra{\ss}e 8--10,
  1040 Wien, Austria
}
\email{dino.rossegger@tuwien.ac.at}
\urladdr{https://drossegger.github.io/}
\thanks{The first, second and fourth authors were supported by the Austrian Science Fund (FWF) 10.55776/P36781. The third author was supported by a Sloan Research Fellowship and by NSF grants DMS-2452105 and DMS-2419591. The last author was supported by the Austrian Science Fund (FWF) 10.55776/PIN1878224 and 10.55776/PAT4699225.}
\subjclass[2020]{%
  Primary 03C57, 03E15;
  Secondary 05C63, 05C60
}

\keywords{%
  Forbidden subgraphs,
  bi-interpretability,
  Borel reducibility,
  $\Sigma$-small,
  computable embeddability condition}

\begin{abstract}
We study classes of countable graphs where every member does not contain
a given finite graph as an induced subgraph---denoted by $\free\G$ for a
given finite graph $\G$. Our main results establish a structural dichotomy
for such classes: If $\G$ is not an induced subgraph of $\path4$, then
$\free\G$ is on top under effective bi-interpretability, implying that the
members of $\free\G$ exhibit the full range of structural and computational
behaviors. In contrast, if $\G$ is an induced subgraph of $\path4$, then
$\free\G$ is structurally simple, as witnessed by the fact that every
member satisfies the computable embeddability condition. This dichotomy
is mirrored in the finite setting when one considers combinatorial and
complexity-theoretic properties. Specifically, it is known that
$\finfree\G$ is complete for graph isomorphism and not a well-quasi-order
under embeddability whenever $\G$ is not an induced subgraph of $\path4$,
while in all other cases $\finfree\G$ forms a well-quasi-order and the
isomorphism problem for $\finfree\G$ is solvable in polynomial time.
\end{abstract}

\maketitle

\tableofcontents

\section{Introduction}
Given a collection $\S$ of finite graphs, let
\[\free{\S}=\{\H : \text{no $\G \in \S$ is an induced subgraph of } \H\}\] 
be the set of $\S$-free graphs. Such classes of graphs are well-studied in combinatorics, see, e.g., Section 7 of the survey \cite{GraphClassesSurvey}. This is the setting of the Erd\"os-Hajnal conjecture \cite{ErdosHajnal,ErdosHajnal2,ErdosHajnalSurvey} and the strong perfect graph theorem \cite{StrongPerfectGraphTheorem}, for example, and many natural graph classes arise in this way. The bipartite graphs are the $\S$-free graphs where $\S$ is the collection of odd cycles and the chordal graphs are the $\S$-free graphs where $\S$ is the collection of cycles of length four or more. Other classes of graphs defined in such a way include the perfect graphs, line graphs of graphs, and split graphs. In combinatorics, graphs are usually assumed to be finite. Given a collection $\S$, a general paradigm for solving questions about the $\S$-free graphs is to show that the $\S$-free graphs are structured, often by admitting some kind of decomposition. Given such structure we can for example count the number of $\S$-free graphs on $n$ vertices or develop efficient algorithms for solving problems such as the colourability or isomorphism problems. On the other hand, sometimes we can show that the $\S$-free graphs have no more structure than graphs in general. The bipartite graphs are an example of an unstructured class of graphs, and this can be shown by the fact that any graph $\G$ can be transformed into a bipartite graph $\H$ by bisecting each edge.

In this paper we will restrict to the case of a single forbidden induced subgraph $\G$, and write $\free \G$ for the $\G$-free graphs. For finite graphs, this case is well-understood in combinatorics with a structural dichotomy centered around the graph $\path{4}$, the path consisting of four vertices. $\free{\path{4}}$ enjoys several equivalent characterizations; the one most relevant to our purposes is its equivalence with the class of {complement-reducible graphs}, in short, \emph{cographs}, introduced in~\cite{complement}. Such graphs have natural applications, e.g., in bioinformatics they naturally model certain binary relations between genes~\cite{orthology}. In the finite case, cographs are defined as the smallest class of graphs that contain a single-vertex graph and are closed under disjoint union and complementation. This allows us to give a recursive definition, yielding a natural representation of any cograph as a labeled tree known as its \emph{decomposition tree} (see \Cref{sec:prel}).

As shown in~\cite{complement}, this tree-based decomposition makes it possible to solve, in polynomial time, a number of problems for finite graphs that are hard or even $\mathsf{NP}$-complete in general. These include, for example, problems such as graph isomorphism, Hamiltonicity, clique detection, and graph coloring. Structural dichotomies for classes of graphs of the form $\free \G$ tend to center around $\free{\path{4}}$. For example, if $\G$ is an induced subgraph of $\path{4}$ (so that $\free \G \subseteq \free {\path{4}}$) then graph isomorphism for $\free \G$ is solvable in polynomial time, and otherwise if $\G$ is not an induced subgraph of $\path{4}$ (so that $\free \G \nsubseteq \free {\path{4}}$) then graph isomorphism for $\free{\G}$ is as hard as graph isomorphism in general\footnote{Note that the graph isomorphism problem is neither known to be $\mathsf{NP}$-complete nor known to lie in $\mathsf{P}$. Babai~\cite{babai2016graph} has shown that graph isomorphism is solvable in quasipolynomial time (see also the expository article~\cite{Helfgott} and its translation~\cite{Helfgott-translation} which contains corrections by Babai to the original proof). Consequently, $GI$ is widely believed to be an intermediate class between the two. Assuming that $GI\notin P$, the two cases are distinct in computational complexity.} ($GI$-complete, see~\cite{booth1979problems}). Similarly, $\finfree{\G}$ is a {well-quasi-order} under the induced subgraph relation if and only if $\G$ is an induced subgraph of $\path{4}$~\cite{damaschke}. Thus, the cographs are the dividing line between structure and non-structure:  $\free \G$ is tractable if and only if $\free \G$ consists entirely of cographs.

In this paper we study classes of countably infinite graphs. Our perspective is that of computable structure theory and descriptive set theory, branches of mathematical logic which provide a framework---based on algorithms and, respectively, topology---for analyzing the complexity of countably infinite structures. When considering classes of structures, we typically observe two extremes. On one side are classes subject to strong global structural constraints, which tightly restrict the behavior of their structures and are thus structurally simple. On the other side lie classes that are on top with respect to some notion of reducibility. These classes accommodate the full spectrum of behaviors and serve as universal benchmarks in complexity comparisons. These classes are thus structurally complex.

As in the finite case, in the infinite case we establish $\free{\path{4}}$ as the main character of structural dichotomies for classes of graphs of the form $\free \G$. We show the following results:
\begin{itemize}
    \item When $\G$ is an induced subgraph of $\path{4}$, $\free{\G}$ is structurally simple:
    \begin{itemize}
        \item $\free{\G}$ is $\Sigma$-small, which means that there are at most countably many existential types sharply realized in $\free{G}$.
        \item $\free{\G}$ satisfies the computable embeddability condition, which means that each existential type sharply realized in a member of $\free{G}$ is computable.
        \item $\free{\G}$ is not on top for bi-interpretability, which means that there is a graph which is not bi-interpretable with a $\G$-free graph. (In particular, there is an automorphism group of a graph which is not the automorphism group of a $\G$-free graph.)
    \end{itemize}
    \item When $\G$ is not an induced subgraph of $\path{4}$, $\free{\G}$ is structurally complex: 
    \begin{itemize}
        \item $\free{\G}$ is not $\Sigma$-small.
        \item $\free{\G}$ does not satisfy the computable embeddability condition.
        \item $\free{\G}$ is on top for bi-interpretability: Every graph is (computably, uniformly, $\Sigma_1$)-bi-interpretable with a $\G$-free graph.
    \end{itemize}
\end{itemize}
This lines up with the dichotomies for finite graphs, summarized by the following table.

\begin{table}[H]
\centering 
\resizebox{\textwidth}{!}{%
\begin{tabular}{|c|c|c|c|} \hline
 & Complexity theory&  Combinatorics &Computable structure theory\\ \hline
$\free{\G} \not\subseteq \free{\path{4}}$
   &  $\finfree{\G}$ \text{is GI-complete}&$\finfree{\G}$ is not a wqo&$\free{\G}$ is on top for eff.\ bi-int.\ \\ \hline
  $\free{\G} \subseteq \free{\path{4}}$& GI on $\finfree{\G}$ is in $\mathsf{P}$& $\finfree{\G}$ is a wqo&$\free{\G}$ is $\Sigma$-small  \\
&& &every $\H$  in $\free{\G}$ has the c.e.c.\\
&& &$\free{\G}$ \text{is not on top for inf.\ bi-int.}\\ \hline
\end{tabular}}
\caption{\label{table:dichotomies}Summary of the dichotomies for classes of graphs obtained by forbidding a finite induced subgraph. The first row indicates the areas of study from which the complexity notions under consideration originate. The last two rows represent mutually exclusive alternatives: given a finite graph $\G$, either $\free{\G}$ satisfies all conditions in the second row, or all conditions in the third row.
}
\end{table}

We do find a new phenomenon for infinite graphs which does not appear for finite graphs. A corresponding notion to the graph-isomorphism problem, and GI-completeness, for infinite structures comes from Borel reducibility theory. We say that a class of structures $\C$ is on top for Borel reducibility if there is a Borel map $\Phi$ such that for all graphs $\G,\H$,
\[ \G \cong \H \Longleftrightarrow \Phi(\G) \cong \Phi(\H).\]
The reader unfamiliar with descriptive set theory should think of a Borel map as a ``reasonably definable'' map. If $\C$ is on top for Borel reducibility, this means that the isomorphism problem for $\C$ is at least as hard as the isomorphism problem for graphs (which is as hard as an isomorphism problem can be). If a class of structures is on top for bi-interpretability, then it is on top for Borel reducibility, but not vice versa. Well-known examples of classes which are on top for Borel reducibility, but not for bi-interpretability, are trees and linear orders.\footnote{Friedman and Stanley \cite{FS} showed that trees and linear orders were on top for Borel reducibility; see Section 1.5 of \cite{treeoftuples} for some discussion of why trees are not on top for bi-interpretability, with the case of linear orders being similar.} We show that cographs can be added to this list.

\begin{manualtheorem}{\ref{theorem:p4borelnotbiinterpret}}
    $\free{\path{4}}$ is on top for Borel reducibility but not for effective bi-interpretability.
\end{manualtheorem}

\noindent In particular, out of all classes $\free \G$, only $\free {\path{4}}$ is on top for Borel reducibility but not on top for effective bi-interpretability.

This phenomenon shows up in the infinite case because deciding whether two infinite trees are isomorphic is very hard (trees are Borel complete) while deciding whether two finite trees are isomorphic is very easy (in linear time by the AHU algorithm). Given a finite cograph $\G$, the size of its decomposition tree $\T$ is polynomial in the number of vertices, and so we get this difference of behaviour. (We note that Friedman and Stanley's argument in~\cite{FS} that trees are Borel complete is by associating to some graph $\G$ the tree of all finite tuples from that graph. While for a countable graph this yields a countable tree, for a finite graph the size of the tree is exponential in the size of the graph, and so even though determining isomorphism of the trees is linear in the size of the tree it is exponential in the size of the graphs.)

For finite graphs, as far as the authors are aware, whenever the isomorphism problem for a class of graphs is shown to be maximally complicated it is always because there is an interpretation of arbitrary graphs into that class. Is this necessarily the case, or is there, as for infinite graphs, some other strategy for showing that isomorphism on a class is maximally complicated?

\begin{question}
    Is there a collection $\S$ of finite graphs such that $\free \S$ is GI-complete, but there is no interpretation from the class of all graphs to $\free \S$?
\end{question}

\noindent More generally, we could ask a similar question for more general classes, e.g., for any hereditary class of finite structures in a relational language.

The paper is organized as follows. In \Cref{sec:prel} we give all the necessary preliminaries and notations regarding (classes of) structures, in particular for labeled trees and graphs. \Cref{sec:cecsigmasmall}, after giving the necessary preliminaries around the notion of $\Sigma$-smallness, is devoted to proving \Cref{theorem:graphsembeddable}. In addition we will prove that $\free{\path{4}}$ is not $\Sigma_2^{\mathsf{inf}}$-small, i.e., it has uncountably many $\Sigma_2^{\mathsf{inf}}$-types (\Cref{prop:notsigmasmall}).   \Cref{sec:results:ontop} is devoted to proving \Cref{theorem:universalcondition} and \Cref{theorem:characterizationborelcomplete}.

\section{Notation and background}
\label{sec:prel}
Structures will be denoted by calligraphic letters such as $\A, \B, \mathcal{C}, \dots$, with their corresponding domains being $A, B, C, \dots$. Embeddings between structures will be denoted via \lq\lq $\hookrightarrow$\rq\rq. 
 All of our countably infinite structures have domain $\omega$ and, given a signature $\tau$, we denote by $Mod(\tau)$ the space of 
$\tau$-structures with domain $\omega$. Notice that such a space is \define{Polish}, i.e., a separable completely metrizable topological space. The basic clopen sets are determined by atomic and negated atomic formulas.

The structures we will be interested in are {labeled trees} and {graphs}. Graphs $\G$ will be irreflexive and undirected. Our notion of subgraph will always be an induced subgraph. We write $\path{n}$ for the \emph{path graph} with $n$ vertices (i.e., a graph forming a linear chain of size $n$), and $\cycle{n}$ for the \emph{cycle graph} with $n$ vertices. We denote by $\overline{\G}$ the complement of $\G$, i.e., the graph having the same set of vertices and such that two vertices have an edge in $\overline{\G}$ if and only if they do not have an edge in $\G$. Given two graphs $\G$ and $\H$ the \define{disconnected union} of $\G$ and $\H$ is defined as the graph consisting of $\G$ and $\H$ with no edge connecting a vertex of $G$ and a vertex of $H$. The \define{connected union} of $\G$ and $\H$ is defined as the graph consisting of $\G$ and $\H$ and in which every vertex of $G$ is connected to every vertex of $H$. Note that the disconnected union of $\G$ and $\H$ is the complement of the connected union of $\overline{\G}$ and $\overline{\H}$.

As mentioned in the introduction, the $\path{4}$-free graphs are crucial to this paper. The finite $\path{4}$-free graphs are exactly the \emph{complement reducible graphs (cographs)}. These are the smallest class of graphs which may be constructed by starting with the single-vertex graphs and closing under connected unions and disconnected unions. For countably infinite $\path{4}$-free graphs we must also use an operation called by~\cite{siblings} a \define{sum over a labeled chain} which is defined as follows.
Let $(\G_i)_{i \in I}$ for some countable set $I$ be a sequence of graphs. Let $\leq$ be a linear order defined on $I$ and let $\ell$ be a labelling function from $I$ to $\{0,1\}$. The labeled sum of $(\G_i)_{i \in I}$ is the graph
obtained by first taking the disconnected union of $(\G_i)_{i \in I}$ and then, adding an edge between $v \in \G_i$ and $w \in \G_j$ for $i < j$ if and only if $\ell(i)=1$. The argument that this cannot introduce any $\path{4}$ takes several cases, e.g., if there is a path $v_1 - v_2 - v_3 - v_4$ with each $v_k$ in a distinct $\G_i$, then one of these has the least $i$, and has either an edge to the other three (if $\ell(i) = 1$) or no edge to any of the others (if $\ell(i) = 0$), a contradiction. 
Following standard convention, we continue to call the infinite $\path{4}$-free graphs cographs.

Since cographs are constructed recursively from smaller building blocks, we can understand them by a corresponding \emph{decomposition tree} which records the construction. A \define{tree} $\T = (T,\preceq)$ is a partially ordered set such that for every $v \in T$, 
$pred(v)=\{x : x \preceq_\T v\}$ is linearly ordered and any two elements have a lower bound. We will call elements of the tree \define{nodes} and a maximal element of the tree a \define{leaf}. A tree is a \define{simple tree} if it has a least element (the \define{root}) and every element is at finite distance from the root, with every element other than the root having an immediate predecessor, in which case we can think of the tree as an acyclic graph with a distinguished element. Every finite tree is a simple tree, and so for finite cographs only simple trees will be required. For infinite cographs we will make use of arbitrary trees. Notice that in a finite tree every two nodes $v,w$ have a greatest lower bound, which we call their \define{meet} and denote by $v\wedge w$.

Given a set of labels $L$, an $L$-labeled tree is a tree $\T$ together with a labelling function $\ell: T \to L$.

We begin by describing the decomposition trees of finite cographs since this is the simpler case. {Finite cographs} are those graphs that can be obtained from the one-vertex graph iterating the operations of connected and disconnected union. This leads to the representation of a cograph $\G$ as a finite $\{0,1,2\}$-labeled tree $\T_\G$ describing how $\G$ was constructed. All leaves of $\T_\G$ will be labeled by $2$ and correspond to the vertices of $\G$. The nodes of $\T_\G$ that are not leaves are labeled by $0$ or $1$, denoting respectively the operations of disconnected and connected union (of the cographs represented by the subtrees of that node). An immediate consequence of this is that two vertices are connected in $\G$ if and only if their meet in $\T_\G$ is labeled by $1$. To guarantee that every cograph $\G$ is  determined by a unique decomposition tree $\T_\G$ we require that if a vertex $t \in T_\G$ is labeled by 1 (0) then all immediate successors of $t$ not being leaves are labeled 0 (1) (see~\cite[Section 3]{damaschke}). Any cograph has a tree of this form, and given any tree of this form it is the decomposition tree of a graph. It is worth noticing, as it will come in handy later, that every node in a finite decomposition tree that is not a leaf has at least two immediate successors. 

For infinite cographs we must account for the operation of sums of labeled chains described previously. We follow~\cite{siblings} which in turn follows~\cite{courcelle2008modular}. We make the following definitions some of which trivialize for simple trees. We say that a tree $\T$ is a \define{meet-tree} if every pair of elements $x,y \in \T$ have a \define{meet} $x \wedge y$ which is the closest common ancestor of $x$ and $y$. Note that the meet of any finite number of elements in a meet-tree is actually the meet of exactly two elements. A meet-tree $\T$ is \define{ramified} if every element of $\T$ is the meet of a
finite set of maximal elements of $\T$. Given a ramified meet-tree $\T$, a labelling of $\T$ is a map $\ell \colon \T \to \{0,1,2\}$ such that the nodes of $\T$ labeled $2$ are exactly the leaves. The labelling is \define{dense} if for any $x \succ y$ non-leaves there is some $z$ with $x \succ z \succeq y$ and $\ell(z) \neq \ell(x)$. Then the decomposition trees are exactly the densely labeled ramified meet-trees. In particular:
\begin{itemize}
    \item Given a densely labeled ramified meet-tree $\T$, let $\G(\T)$ be the graph whose vertices are the leaves of $\T$, with two vertices $u$ and $v$ being joined by an edge if and only if $u \wedge v$ is labeled $1$.
    \item Given a cograph $\G$, to define the \define{decomposition tree} $\T(\G)$ of $\G$ we must first introduce some definitions. A \define{module} of a graph $\G$ is any subset $A \subseteq G$ such that for all $u,v \in A$ and $w \notin A$, there is an edge between $u$ and $w$ if and only if there is one between $v$ and $w$. A module is \define{strong} if it is either comparable to or disjoint from every module. A module is \define{robust} if it is the least strong module containing two vertices of $\G$. For this paper it mostly suffices to note that robust modules are graph-theoretically defined and the property of being a robust module is maintained by isomorphisms. We refer the reader to~\cite{siblings} for more details.
    
    The \define{decomposition tree} $\T(\G)$ of $\G$ is the tree whose nodes are $R(\G)$, the \define{robust modules} of $\G$.  The trivial modules---which consist of a single vertex---get the label $2$ and are the leaves of $\T$. If $A \subseteq G$ is the least robust module of $\G$ containing $u$ and $v$, then we give $A$ the label $1$ if there is an edge between $u$ and $v$, and $0$ if there is no such edge; this does not depend on the choice of $u$ and $v$.
\end{itemize}
Theorem 6.19 of~\cite{siblings} proves that these constructions give a one-to-one correspondence between cographs and ramified meet-trees densely labeled by $\{0,1,2\}$. In Theorem \ref{theorem:cotreescographsBBF} we show that they are in fact bi-interpretable in infinitary logic. This is somewhat implicit in~\cite{courcelle2008modular} though Courcelle and Delhomm\'e use second-order logic. In particular, an automorphism of the decomposition tree induces an automorphism of the cograph by restricting to the leaves, and an automorphism of the cograph induces an automorphism of its tree by acting on the leaves and extending to the whole tree in a unique way.

\section{\texorpdfstring{$\free{\G}$}{Free(G)} that are structurally simple}
\label{sec:cecsigmasmall}
Recall that the \define{existential type} of a tuple $\bar a$ in a structure $\A$ is the set
 \[\Etp{\A}{\bar a}=\{\phi : \phi(\overline{x})\text{ is a first-order existential formula with }\A \models \varphi(\overline{a})\}.\]

A first-order type $p$ is realized in a structure $\A$ if there is a tuple $\bar a$ such that for all $\phi\in p$, $\A\models \phi(\bar a)$, and we say that $p$ is \define{realized} in a class of structures $\K$ if there is a structure in $\K$ realizing $p$. 

Furthermore, an existential type $p$ is \define{sharply realized} in $\K$ if
there exists a tuple $\bar{a}$ in some structure $\A \in \K$ so that $p=\Etp{\A}{\bar a}$, and it is \define{sharply realized} in $\A$ if it is sharply realized in $\K=\{\A\}$.

 The following definition first appeared under the name recursive embeddability condition in~\cite{richter}.
\begin{definition}
\label{definition:computablembeddabilitycondition}
    A structure $\A$ satisfies the \define{computable embeddability condition (c.e.c)} if each existential type sharply realized in $\A$ is computable. A class $\K$ satisfies the computable embeddability condition if every structure in $\K$ satisfies it.
\end{definition}

The following classes have the computable embeddability condition: linear orderings, trees (\cite{richter}), Boolean algebras (\cite{harris2012}), $\mathbb{Q}$-vector spaces, algebraically closed fields, differentially closed fields, abelian $p$-groups, and equivalence structures.

Montalbán~\cite{counting} defined the following non-effective version of the c.e.c. Though his definition was in terms of universal types, we state it here in terms of existential types. The two definitions are equivalent.  
\begin{definition}
\label{definition:sigmasmall}
    A class $\K$ of structures is \define{$\Sigma$-small} if there are at most countably many existential types sharply realized in $\K$.
\end{definition}

\noindent It is an immediate observation that a class of structures is $\Sigma$-small if and only if there is an $X\in 2^\omega$ so that all structures in $\K$ have the computable embeddability condition relative to $X$.

Any three element connected induced subgraph of a graph in $\free{\path{3}}$ must be complete and thus the edge relation on such a graph must be transitive. From this it follows that graphs in $\free{\path{3}}$ are just disconnected unions of cliques.
It also shows that graphs in $\free{\path{3}}$ are definitionally equivalent with equivalence structures: Given an equivalence relation we can define a graph in $\free{\path{3}}$ by putting an edge between $x$ and $y$ if and only if $x \neq y$ are equivalent, and given a graph in $\free{\path{3}}$ we can reverse this by putting $x$ equivalent to $y$ if and only if $x = y$ or there is an edge between $x$ and $y$. 
Using this equivalence we get the following from results about equivalence structures in the literature~\cite[Section X]{within}.
\begin{proposition}
\label{proposition:P3free}
Every graph in $\free{\path{3}}$  (i.e., equivalence structure) has the computable embeddability condition. Thus, $\free{\path{3}}$ is $\Sigma$-small.
\end{proposition}

We are going to prove the same result for $\free{\path{4}}$. To do so, we use a technique that appeared in~\cite[Chapter X.2]{within} and is based on~\cite{richter}. 

Let $\tau$ be a finite vocabulary without function symbols, and let $\K$ be a class of $\tau$-structures.  Given a set $A$, define $\tau_A$ by increasing
the vocabulary $\tau$ with new constant symbols, one for each element of $A$. Then, a $\tau_A$-structure is determined by a $\tau$-structure $\B$ and a map $f : A \rightarrow \B$ describing the assignments of the new constants. We denote by 
$\finitary{\K}{}$ the set of finite substructures in $\K$ and, given some
$\A \in \finitary{\K}{}$, let \[\K_\A=\{\B\in\K: \B \text{ is a } \tau_A\text{-structure and } \A \overset{g}{\hookrightarrow} \B \text{ for some $\tau_A$-embedding $g$}\}.\]

Recall that a \define{well quasi-order} (in short, \define{wqo}) is a quasi-order without infinite antichains and infinite descending chains. We use the following sufficient condition for a class to be $\Sigma$-small.

\begin{theorem}[{\cite[Theorem 10.2.2]{within}}]
\label{theorem:wqocondition}
   Let $\K$ be a class of structures in finite vocabulary without function symbols. Suppose that for every finite substructure $\A \in \finitary{\K}{}$, $\K_\A^{fin}$ is a wqo under the embeddability relation. Then, in $\K$, every $\Sigma_1^{\mathsf{inf}}$ formula is equivalent to a finitary existential formula. In particular, $\K$ is $\Sigma$-small.
\end{theorem}

To prove that such a class is $\Sigma$-small we want to apply \Cref{theorem:wqocondition}. To do so, we will use the following result on labeled graphs. Let $(W,\leq)$ be a quasi-order. A $W$-labeled graph is a graph $\G$ together with a map $l:G\to W$, and $\G$ is a \define{labeled induced subgraph} of $\G'$ if there is an induced-subgraph embedding $f:\G\to\G'$ such that $l(v)\leq l'(f(v))$ for every $v\in G$. By~\cite[Theorem 2]{labelledwqo}, finite $\path{4}$-free graphs are a wqo under the labeled induced subgraph relation whenever the set of labels is a wqo.

\begin{theorem}
\label{lemma:cographswqo}
The class $\free{\path{4}}$ is $\Sigma$-small.
\end{theorem}
\begin{proof}
Let $\G\in\finitary{\free{\path{4}}}{}$ with domain
$G=\{a_1,\ldots,a_n\}$. Given
$\H\in\finitary{\free{\path{4}}}{\G}$, label the interpretation of
$a_i$ by $i$, for each $1\leq i\leq n$, and give all other vertices a
new label $0$. Regard $\{0,1,\ldots,n\}$ as an antichain.

By~\cite[Theorem 2]{labelledwqo}, these labeled finite
$\path{4}$-free graphs form a wqo under the labeled induced subgraph
relation. Since the labels are pairwise incomparable, every such
embedding preserves the interpretations of the constants from $\G$.
Hence $\finitary{\free{\path{4}}}{\G}$ is a wqo and, by \Cref{theorem:wqocondition}, this implies that
$\free{\path{4}}$ is $\Sigma$-small.
\end{proof}

We also obtain that any graph in $\free{\path{4}}$ has the computable embeddability condition. This follows from the following more general theorem. For simplicity we consider only classes of graphs, though it holds in general for finite relational languages.

\begin{theorem}
\label{theorem:wqoimpliescec}
Let $\K$ be such that for every finite $\B \in \finitary{\K}{}$, $\finitary{\K}{\B}$ is a wqo and membership in $\finitary{\K}{\B}$ is computable. Then, every graph $\A \in \K$ has the computable embeddability condition.
\end{theorem}

In particular every $\G \in \free{\path{4}}$ has the computable embeddability condition.

\begin{proof}
  Let $\A \in\K$ and $\bar a\in A^{<\omega}$. Let $\B$ be the induced subgraph with vertices $\bar{a}$. First consider the set of $\H \in \finitary{\K}{\B}$ which do not embed in $\A$ over $\B$. This is upwards closed, and so as $\finitary{\K}{\B}$ is a wqo there is a finite list $\J_1,\ldots,\J_n \in \finitary{\K}{\B}$ such that if $\H$ does not embed into $\A$ over $\B$, some $\J_i$ embeds in $\H$ over $\B$.  Thus, given some $\H \in \finitary{\K}{\B}$, it is computable to decide whether $\B$ extends to a copy of $\H$ in $\A$: Simply check whether there is some $\J_i$ which embeds in $\H$ over $\B$. Thus we can compute which finite substructures $\B$ can be extended to in $\A$. This is the same as determining which existential formulas $\bar{a}$ satisfies in $\A$.
\end{proof}

A natural question is whether the class of cographs is not only $\Sigma$-small but also small for types of higher complexities. For example, say that a class $\K$ is \define{$\Sigma^{\mathsf{inf}}_2$-small}, if there are only countably many different complete $\Sigma^{\mathsf{inf}}_2$-types sharply realized in $\K$. We will now show that cographs are not $\Sigma^{\mathsf{inf}}_2$-small.

\begin{proposition}
\label{prop:notsigmasmall}
  $\free{\path{3}}$ is not $\Sigma_2^{\mathsf{inf}}$-small. In particular, $\free{\path{4}}$ is not $\Sigma_2^{\mathsf{inf}}$-small as well.
\end{proposition}
\begin{proof}
Recall that $\free{\path{3}}$ can be viewed as equivalence structures and the latter is not $\Sigma_2^{\mathsf{inf}}$-small~\cite[Section 4.2]{counting}. Since $\path{3}$ is an induced subgraph of $\path{4}$ we have $\free{\path{3}} \subseteq\free{\path{4}}$ and hence, by the fact that the property of being $\Sigma_2^{\mathsf{inf}}$-small is closed under subsets, $\free{\path{4}}$ cannot  be $\Sigma_2^{\mathsf{inf}}$-small.
\end{proof}

The main result of this section is that $\free{\path{4}}$ is a dividing line for being $\Sigma$-small and having the computable embeddability condition.

The following observation can be proved by chasing definitions.
\begin{lemma} \label{lemma:forbiddenGH}
    A finite graph $\G$ is an induced subgraph of finite graph $\H$ if and only if $\free{\G}\subseteq \free{\H}$.
\end{lemma}

\begin{theorem}
\label{theorem:graphsembeddable}
For a finite graph $\G$, the following are equivalent:
\begin{enumerate}
    \item[(i)] every graph in $\free{\G}$ has the computable embeddability condition;
    \item[(ii)] $\free{\G}$ is $\Sigma$-small;
    \item[(iii)] $\free{\G}\subseteq \free{\path{4}}$.
\end{enumerate}
\end{theorem}
\begin{proof}
That (i) implies (ii) is trivial.

The fact that (iii) implies (i) follows from the fact that every graph in $\free{\path{4}}$ has the computable embeddability condition.
   To conclude the proof, we show that (ii) implies (iii) proving its contrapositive. By \Cref{lemma:forbiddenGH} we have that $\G$ is not an induced subgraph of $\path{4}$. It is well-known in graph theory that if $\G$ is not an induced subgraph of $\path{4}$, then either $\G$ or $\bar{\G}$ contains a cycle (see e.g., the first theorem of Section 4.7 in~\cite{booth1979problems}).
Hence, take the maximum $m$ such that $\cycle{m}$ is an induced subgraph of $\G$ or $\overline{\cycle{m}}$ is an induced subgraph of $\G$.
   
Damaschke~\cite{damaschke} showed that the classes $\free{\G}$ for such graphs $\G$ are not wqos. Indeed, $\{\cycle{m+1+n} : n \in \omega\}$ is an antichain in $\free{\G}$ if $\cycle{m}$ is an induced subgraph of $\G$ and   
$\{\overline{\cycle{m+n+1}} : n \in \omega\}$ is an antichain if $\overline{\cycle{m}}$ is an induced subgraph of $\G$.  For every $i$, let $\phi_i$ (respectively $\psi_i$) be the existential formula saying that there is an induced subgraph isomorphic to $\cycle{m+1+i}$ (respectively $\overline{\cycle{m+1+i}}$).

For every infinite $I \subseteq \omega$:
\begin{itemize}
    \item if $\cycle{m}$  is an induced subgraph of $\G$, let $\A_I$  be the graph in $\free{\G}$ defined as the disconnected union of all the $\cycle{m+1+i}$'s for $i \in I$.
    \item if $\overline{\cycle{m}}$ is an induced subgraph of $\G$, let $\B_I$ be the graph in $\free{\G}$ defined as the connected union of all the $\overline{\cycle{m+1+i}}$'s for $i \in I$. 
\end{itemize}

Notice that $\A_I \models \phi_i$ if and only if $i \in I$ (similarly for $\B_I$). 
Also,  $\Etp{\A_I}{\langle\rangle} \supseteq \{\phi_i : i \in I\}$ and $\Etp{\B_I}{\langle\rangle} \supseteq \{\psi_i : i \in I\}$ where $\langle\rangle$ denotes the empty tuple.  There are uncountably many existential types over the empty tuple, so $\free{\G}$ is not $\Sigma$-small.
 \end{proof}

Most natural classes that are known to be $\Sigma$-small allow for a finite set of $\Sigma_1^{\mathsf{inf}}$ or $\Pi_1^{\mathsf{inf}}$ relations so that after adding these relations to the structures, every $\Pi_1^{\mathsf{inf}}$-definable relation becomes $\Sigma_1^{\mathsf{inf}}$-definable. These augmented structures are often called the \lq\lq jump of the structure\rq\rq. For example, for linear orderings one obtains the jump by adding the successor( and relations for first and last elements) and for Boolean algebras the jump is obtained by adding the atom relation. 
  \begin{question}
    Which canonical set of relations needs to be added to obtain the jump of $\free{\path{4}}$?
  \end{question}

\section{\texorpdfstring{$\free{\G}$}{Free(G)} that are structurally rich}
\label{sec:results:ontop}

The following definitions are standard in descriptive set theory, see e.g.,~\cite{gao2008invariant,kechris2012classical}. For a Polish space $X$, the class of \define{Borel sets} is given by the smallest $\sigma$-algebra containing all open subsets of $X$.
Given two Polish spaces $X,Y$, a function $f : X \rightarrow Y$ is \define{Borel} if for every open set $U \subseteq Y$, $f^{-1}(U)$ is Borel. 

\begin{definition}[\cite{FS}]
\label{definition:borelreducibility}
Given two classes closed under isomorphism $\C\subseteq Mod(\tau)$ and $\D\subseteq Mod(\tau')$ for two languages $\tau$ and $\tau'$ we will say that $\C$ is \define{Borel reducible} to $\D$ if there is a Borel map $\Phi : \C \to \D$ such that for $\A,\B \in \C$,
\[ \A \cong \B \Longleftrightarrow \Phi(\A) \cong \Phi(\B).\]
This means that isomorphism on $\D$ is at least as complicated as isomorphism on $\C$.
\end{definition}
\begin{definition}
    A class of structures $\C$ is \define{on top for Borel reducibility} if for every other class of structures $\D$ in any vocabulary $\tau$, $\D$ Borel reduces to $\C$.
\end{definition}
Classes on top for Borel reducibility are often called $S_\infty$-complete in the literature as every equivalence relation induced by a Borel action of a closed subgroup of $S_\infty$ reduces to it. They are also called Borel complete; we do not use this terminology as it gives the impression that the isomorphism relation on the class is Borel which is not the case.
\begin{theorem} \label{theorem:cographborelcomplete}
   $\free{\path{4}}$ is on top for Borel reducibility.
\end{theorem}
\begin{proof}
To prove the claim we reduce the class of simple trees (which is well-known to be on top for Borel reducibility, see~\cite{FS}) to the subclass of infinite \emph{simple} decomposition trees (recalling that these are the rooted decomposition trees such that every non-root node has an immediate predecessor). Then we will easily see that it is Borel to produce from a decomposition tree the corresponding cograph.

Given a tree $\T$, we produce a tree $\T_\ell$ adding to every node of $\T$ infinitely many new children each of which is a leaf.
The root of $\T_\ell$ is labeled by $1$ and for every $v \in \T_\ell$, if $v$ is a leaf, the label of $v$ is $2$ while if $v$ is not a leaf and has odd (even) distance from the root, the label of $v$ is $0$ ($1$). The construction implies that $\T_\ell$ is an infinite labeled tree. Because each non-leaf node of $\T_\ell$ has at least two leaf children, $\T_\ell$ is ramified. The definition of the labels ensures that the labelling is dense. Thus $\T_\ell$ is the decomposition tree of an infinite cograph.

We argue that $\T \cong \T'$ if and only if $\T_\ell \cong \T_\ell'$. If $\T \cong \T'$, we can extend the isomorphism witnessing $\T \cong \T'$ to take care of the new leaves that have been added in the obvious way. If $\T_\ell \cong \T_\ell'$  we can just remove all of the leaves to obtain an isomorphism $\T \cong \T'$.

The standard way to produce from a tree the corresponding cograph is Borel. Recall that this was done as follows. We first identify the leaves which we take as the vertices of the cograph. Given vertices $u,v$, we can find in the tree the meet $u \wedge v$ of $u$ and $v$, and set an edge between $u$ and $v$ if and only if $u \wedge v$ is labeled $1$.
\end{proof}

Countable trees and graphs are both on top for Borel reducibility. However, as can be seen from work of Gao~\cite{gao2001a}, graphs are structurally much richer than trees. The Borel reduction of some class $\C$ into the class of graphs can be facilitated by the classical first-order interpretation of arbitrary structures in graphs, see e.g.~\cite{marker2002}. However, the reduction given in~\cite{FS} that shows that the class of trees is on top for Borel reducibility is much weaker. Indeed, Gao showed that trees are not \lq\lq faithfully Borel complete\rq\rq~\cite{gao2001a}. Faithful Borel reductions can be seen as a weak form of model-theoretic bi-interpretation, and as indeed every model-theoretic bi-interpretation induces a faithful Borel reduction, this highlights a difference in the structural richness of the two classes.

The strongest notion of reducibility is that of reductions by effective bi-interpretations. If a class $\mathbb{C}$ is reducible by effective bi-interpretations to a class $\mathbb{D}$ then this means that essentially any computability-theoretic behaviour that can be found in the class $\mathbb{C}$ can also be found in $\mathbb{D}$. It was shown in~\cite{computablefunctors} that a structure $\A$ being effectively interpretable in $\B$ is equivalent to there being a computable functor from the isomorphism class of $\B$ to the isomorphism class of $\A$ (and a similar characterization of bi-interpretability in terms of a computable equivalence of categories). We will use this semantic approach in this paper. For an extensive treatment of bi-interpretability, including a proof of the equivalence mentioned above we refer the reader to~\cite{within,beyond}.

The following definitions first appeared in~\cite{computablefunctors}. Some definitions are taken from~\cite{rossegger2022} but are equivalent to definitions that can be found in~\cite{computablefunctors}.
\begin{definition}\label{def:compfunc}
    A \define{computable functor} $F:\C\rightarrow \D$ is witnessed by two computable operators $\Phi$ and $\Phi_*$ such that
    \begin{itemize}
        \item for every $\hat \A \in \C$, $\Phi^{D(\hat\A)}$ is the atomic diagram of $F(\hat \A) \in \D$        \item for every isomorphism $f: \hat \A \rightarrow \tilde \A$ in $\C$, $\Phi_*^{\hat \A \oplus f \oplus\tilde{\A}}=F(f)$. 
    \end{itemize}
\end{definition}
In order to define a semantic notion of reduction equivalent to effective bi-interpretability we need a few more category theoretic definitions.
\begin{definition}[\cite{computablefunctors}]\label{def:effiso}
  A functor $F:\C\to \D$ is \define{effectively isomorphic} to $G:\C\to
  \D$ if there is a Turing operator $\Lambda$ such that for every
  $\A\in \C$, $\Lambda^\A$ is an isomorphism from $F(\A)$ to $G(\A)$, and the following diagram
  commutes for all $\A,\B\in \C$ and every $\gamma\in Hom(\A, \B)$. 
  \begin{center}
  \begin{tikzcd}
    F(\A)\ar[r,"\Lambda^\A"]\ar[d,"F(\gamma)"] & G(\A)\ar[d,"G(\gamma)"]\\
    F(\B)\ar[r,"\Lambda^\B"] & G(\B)
  \end{tikzcd}
\end{center}
\end{definition}
The following definition appeared for arbitrary categories in~\cite{rossegger2022}. Here we only deal with classes of structures where the arrows are given by the isomorphism relation.
\begin{definition}\label{def:cbf}
  We say that $\C$ is \define{CBF-reducible} to $\D$ if
  \begin{enumerate}
    \item there is a computable functor $F:\C \to \D$ and a computable functor $G:\D \supseteq \hat{\D}\to \C$ where $\hat{\D}$ is the isomorphism closure 
  of $F(\C)$,
    \item\label{it:def:cbfpi1}$F\circ G$ is effectively isomorphic to $Id_{\hat{\D}}$ and $G\circ
      F$ is effectively isomorphic to $Id_\C$,
    \item\label{it:def:cbfpi2} and, if $\Lambda_\C$, $\Lambda_\D$ are the operators witnessing
      the effective isomorphism between $G\circ F$ and $Id_\C$,
      respectively, $F\circ G$ and $Id_{\hat{\D}}$, then for every $\A\in \C$,
      $F(\Lambda_\C^\A)=\Lambda_\D^{F(\A)}: F(G(F(\A))) \to F(\A)$ and for every $\B\in\hat{\D}$,
      $G(\Lambda_\D^\B)=\Lambda_\C^{G(\B)}: G(F(G(\B)))\to G(\B)$.
  \end{enumerate}
Functors satisfying \cref{it:def:cbfpi1,it:def:cbfpi2} are said to be \define{effective pseudo-inverses}.
\end{definition}

The main result of~\cite{computablefunctors} shows that CBF-reducibility is equivalent to reduction via effective bi-interpretability. We obtain the following notion of universality for this notion.

\begin{definition}\label{def:ontopeffective}
\label{definition:ontopeffective}
A class of structures $\C$ is \define{on top for effective bi-interpretability} if every Borel class of countable structures CBF-reduces to $\C$.
\end{definition}

If we replace the Turing operators in \cref{def:compfunc,def:effiso,def:cbf} with Borel operators we obtain the notions of \define{Borel functors} and \define{BBF-reducibility}. In~\cite{borelfunctors} it was shown that BBF-reducibility is equivalent to reductions via infinitary bi-interpretability, and thus we can define the notion of being \define{on top for infinitary bi-interpretability} analogous to \cref{def:ontopeffective}. If two structures are infinitary bi-interpretable, then their automorphism groups must be isomorphic~\cite[Theorem 6]{borelfunctors}. Thus, if a class $\C$ is on top for infinitary bi-interpretability, then every automorphism group of countable structures must be found among the automorphism groups of structures in $\C$. 
We will use this fact to show that the class of cographs is not on top for infinitary bi-interpretability (\Cref{theorem:cotreesnotuniversal}).

But first, we show that the classes of cographs and cotrees are BBF-equivalent (\Cref{theorem:cotreescographsBBF}). Using \cref{claim:M,claim:S}, we show that a cograph and its decomposition tree are not only infinitary bi-interpretable but also bi-interpretable in the standard model-theoretic sense (see~\cite[Chapter 5]{Hodges}) which is more restrictive. This suggests that cographs are not even on top for faithful Borel reducibility. (Gao~\cite{gao2001a} has only shown that simple trees and linear orders are not on top for faithful Borel reducibility, but we expect that this extends to all trees and, we expect, to densely labeled ramified meet-trees.)

\begin{question}
    Are cographs on top for faithful Borel reducibility?
\end{question}

\begin{question}
    Are trees (partially ordered sets where the predecessors of any element are linearly ordered) on top for faithful Borel reducibility?
\end{question}

To begin, we must show that certain relations on cographs are definable. Given distinct vertices $u,v$ in a graph, let $M(u,v)$ denote the least module containing $u$ and $v$. We can define $M(u,v)$ by a closure operation: It is the smallest set $M$ containing $u$ and $v$ and such that if $x,y \in M$ and there is an edge from $w$ to $x$ and no edge from $w$ to $y$ then $w \in M$. We say that \define{$w$ witnesses that $\{u,v\}$ is not a module} if $w \notin \{u,v\}$ and there is an edge from $w$ to one of $u$ or $v$ and no edge from $w$ to the other. \cref{claim:M} shows that in a cograph the closure can be done in two steps, allowing us to express $w\in M(u,v)$ in first-order logic.
\begin{lemma} \label{claim:M}
    Let $\G$ be a cograph and let $u,v$ be distinct vertices in $\G$. Then $w\in M(u,v)$ if and only if either:
    \begin{enumerate}
        \item $w = u$ or $w = v$,
        \item $w$ witnesses that $\{u,v\}$ is not a module, or
        \item there is $w'\neq u,v$ such that $w'$ witnesses that $\{u,v\}$ is not a module and $w$ witnesses that either $\{u,w'\}$ or $\{v,w'\}$ is not a module.
    \end{enumerate}
    Therefore $w\in M(u,v)$ can be expressed using an $\exists_1$ formula of finitary first-order logic. Moreover, we will show that
    \[ \tag{$*$}  M(u,v) = \{w: w\land u \succ u\land v\} \cup\{w: w \land v \succ u\land v\}.\]
\end{lemma}
\begin{proof}
    Having shown the rest of the lemma, we can express $w\in M(u,v)$ using an $\exists_1$ formula by simply transcribing (1), (2), and (3). Recall that we said that $w$ witnesses that $\{u,v\}$ is not a module if there is an edge from $w$ to one of $u$ or $v$ and no edge from $w$ to the other. This is quantifier-free, and it is only in (3) that we introduce an existential quantifier.

    For the rest of the lemma, one direction---that if $w$ satisfies (1), (2), or (3) then $w \in M(u,v)$---is clear. We must prove the other direction. One way to argue this is to verify by case analysis that if $w$ enters $M(u,v)$ in three steps, but not in two steps, then there must be an induced subgraph isomorphic to $\path{4}$. We will take a more structural approach via the decomposition tree.

    Now we will show that
    \[  \{w: w\land u \succ u\land v\} \cup\{w: w \land v \succ u\land v\} \subseteq M(u,v).\]
    If $w = u$ or $w = v$ then by (1) we have $w \in M(u,v)$. Otherwise, suppose that $w \land u \succ u \land v$. Then $w \land v = u \land v$. We explain this conclusion in more detail, but leave later derivations to the reader. Since $w \wedge u \succ u \wedge v$, $u$ and $v$ meet on the tree closer to the root than the point where $w$ and $u$ meet. Thus $v$ meets $w$ at the same place that $u$ meets $v$.
\[
\Tree [.{$u \wedge v$}
          [.{$v$} ]
          [.{$w \wedge u$}
              [.{$u$} ]
              [.{$w$} ]
          ]
      ]
\]
    We see that $w \land v = u \land v$. We suggest in later arguments that the reader draw a similar picture if necessary.
    
    Either the labels of $w \land u$ and $w \land v$ are different---in which case there is an edge from $w$ to $u$ and none from $w$ to $v$, or vice versa, and so $w \in M(u,v)$ by (2)---or they are the same. If they have the same label---say, without loss of generality, label $0$---then there is some node $\sigma$ with label $1$ in between them, that is, $w\land u \succ \sigma \succ w\land v$. Since $\T$ is ramified, $\sigma$ is the meet of two leaves extending it. Since $w\land u\succ\sigma$, at least one of these leaves, say $t$, satisfies $t\land(w\land u)=\sigma$. Hence $w\land t=t\land u=\sigma$. We also have $t \land v = w \land v = u \land v$.
    \[ \Tree [.{$u \wedge v$} [.{$\sigma=t \wedge u$} [.{$u \wedge w$} [.$u$ ] [.$w$ ]] [.$\vdots$ [.$t$ ]]] [.$\vdots$ [.$v$ ] ]] \]
    Thus: There is an edge from $t$ to $u$ and no edge from $t$ to $v$, and there is an edge from $w$ to $t$ and no edge from $w$ to $u$ or to $v$. Then $w \in M(u,v)$ by (3).

    Note that every $w$ on the right-hand-side of ($*$) satisfies one of (1), (2), or (3). So if we can show that
    \[   M(u,v) \subseteq \{w: w\land u \succ u\land v\} \cup\{w: w \land v \succ u\land v\} \]
    then we are done. The right-hand-side of ($*$) contains $u$ and $v$ so it suffices to show that it is a module. Suppose that $w$ is on the right-hand-side and that $w'$ is not. Suppose without loss of generality that $w \land u \succ u \land v$. We also have $w' \land u = w' \land v \preceq u \land v$. This implies that $w \land w' = w' \land u = w' \land v$. So whether or not there is an edge between $w$ and $w'$ depends on the label of $w \land w'$, and this node is independent of the choice of $w$.
\end{proof}

Given distinct vertices $u,v$ in a cograph, let $S(u,v)$ be the least strong module containing $u$ and $v$. Note that $S(u,v)$ is robust. The definition of a strong module forces the set $S(u,v)$ to not only contain $M(u,v)$, but also be closed under vertices $w$ not yet in the set that can form an \emph{overlapping} set $M(u',w)$ where $u'$ is in the set and some $v'\notin M(u',w)$. (We define two sets as \define{overlapping} if they are not disjoint and neither is contained in the other.) \cref{claim:S} shows that it is enough to consider the situation where $u'=u$ and $v'=v$. In particular, the closure can be done in one step.
\begin{lemma} \label{claim:S}
    Let $\G$ be a cograph with tree $\T$, and let $u,v$ be distinct vertices in $\G$. Then $w\in S(u,v)$ if and only if either $w\in M(u,v)$, or $w\neq u$ and $v\notin M(u,w)$. Therefore applying \Cref{claim:M}, $w\in S(u,v)$ can be expressed using the disjunction of an $\exists_1$ formula and a $\forall_1$ formula.
\end{lemma}
\begin{proof}
    Having shown the rest of the lemma, the formula expressing $w \in S(u,v)$ is the disjunction of the $\exists_1$ formula expressing $w \in M(u,v)$ and the $\forall_1$ formula expressing $w \neq u$ and $v \notin M(u,w)$.
    
    For the rest of the lemma, in one direction, if $w \in M(u,v) \subseteq S(u,v)$ then $w \in S(u,v)$. And if $w \neq u$ and $v \notin M(u,w)$ then since $M(u,w)$ intersects $S(u,v)$ at $u$, and $S(u,v) \nsubseteq M(u,w)$, we must have $M(u,w) \subseteq S(u,v)$ as otherwise $S(u,v)$ would not be strong, and so $w \in S(u,v)$.

    Now suppose that $w \in S(u,v)$. Following the definition of the decomposition tree, $w \succeq u \land v$. If $w \land u \succ u \land v$ or $w \land v \succ u \land v$ then, by ($*$) of \cref{claim:M}, $w \in M(u,v)$.
    \[ \Tree [.{$u \wedge v$} [.{$w \wedge u$} [.$w$ ] [.$u$ ]] [.$\vdots$ [.$v$ ]]] \qquad \qquad \Tree [.{$u \wedge v$} [.$\vdots$ [.$u$ ]] [.{$w \wedge v$} [.$w$ ] [.$v$ ]] ]\]
    Otherwise $w \land u = w \land v = u \land v$.
    \[\Tree [.{$u \wedge v \wedge w$} [.$\vdots$ [.$u$ ] ] [.$\vdots$ [.$v$ ] ] [.$\vdots$ [.$w$ ]]] \]
    Then, again by ($*$) of \cref{claim:M}, $v \notin M(u,w)$.
\end{proof}

We now use the fact, from \cref{claim:S}, that the relation $w \in S(u,v)$ is definable to give a bi-interpretation between a cograph and its decomposition tree.

\begin{theorem} \label{theorem:cotreescographsBBF}
    A cograph and its decomposition tree are model-theoretically bi-interpretable, with the formulas used independent of the models. In particular, the theory of cographs and the theory of densely labeled ramified meet-trees are model-theoretically bi-interpretable. It follows that these theories are BBF-equivalent.
\end{theorem}
\begin{proof}
Given a cograph $\G$, we define its decomposition tree interpreted in $\G$ as follows. The domain of $\T_\G$ consists of pairs $(u,v) \in G^2$ with $(u,v) \sim (u',v')$ if and only if $(u,v)=(u',v')$, or $u \neq v$, $u' \neq v'$, and the robust module $S(u,v)$ generated by $u,v$ is the same as the one $S(u',v')$ generated by $u',v'$. This can be expressed by the $\forall_2$ formula saying that for all $w$, $w \in S(u,v)$ if and only if $w \in S(u',v')$. We set $(u,v) \preceq (u',v')$ if and only if either (a) $(u,v)\sim(u',v')$, or (b) $u\neq v$ and $u' = v' \in S(u,v)$, or (c) $u \neq v$, $u' \neq v'$, and $S(u,v) \supseteq S(u',v')$. This is also expressed by a $\forall_2$ formula. Finally, we put the label $2$ on all pairs $(u,v)$ with $u = v$, and otherwise we label $(u,v)$ by $0$ if there is no edge from $u$ to $v$, and $1$ if there is an edge from $u$ to $v$. This is quantifier-free.

The interpretation of $\G$ in $\T_\G$ is easier. The domain of $\G$ is the leaves of $\T_\G$, which are the nodes labeled $2$, with the equivalence relation being equality. We put an edge between $u$ and $v$ if and only if $u \wedge v$ is labeled $1$. This is all first-order definable.

Finally, we must show that this is a bi-interpretation. This is easy to see: When $\G$ is interpreted in $\T_\G$ which is interpreted in $\G$, the elements of $\G$ are interpreted as leaves of $\T_\G$, which themselves are interpreted as pairs $(u,u)$ in $\G$. So the isomorphism between these two copies of $\G$ is just the map $u \mapsto (u,u)$ on $\G$ which is of course definable. When $\T_\G$ is interpreted in $\G$ which is interpreted in $\T_\G$, a leaf of $\T_\G$ is represented by a pair $(u,u)$ of $\G$, where $u$ is the corresponding leaf of $\T_\G$, and this is interpreted as $(u,u)$ in $\T_\G$. A non-leaf node, say $w = u \wedge v$, of $\T_\G$ is interpreted as a pair $(u,v)$ of $\G$, which is a pair $(u,v)$ of leaf nodes. This isomorphism is also definable as given $w \in \T$ and $(u,v) \in T^2$, we have $w \mapsto (u,v) / \sim$ if either $w$ is a leaf and $u = v = w$ or $w$ is not a leaf and $u \wedge v = w$.

For the last statement of the theorem, note that our interpretations are uniform in $\G$, and thus we obtain a BBF-reduction from cographs to densely labeled ramified meet-trees via the equivalence proven in~\cite{borelfunctors}. From the fact that the association between a cograph and its decomposition tree gives a bijection between cographs and densely labelled ramified meet-trees~\cite{siblings}, we get that the forward functor in this BBF-reduction is onto up to isomorphism (cf. \cref{def:cbf}: in our case $\hat{\mathbb D}$ is the whole class). Thus, we also obtain an inverse reduction, proving the equivalence. Model theoretic bi-interpretability of the classes now follows from the fact that all formulas in the interpretations are finitary.
\end{proof}

The bi-interpretation just shown will be used to prove the following:

\begin{theorem} \label{theorem:cotreesnotuniversal}
The class of cographs is not on top for infinitary bi-interpretability. 
\end{theorem}
\begin{proof}
    We will show that $\mathbb{Z}_3$ is not the automorphism group of any cograph.\footnote{We note that $\mathbb{Z}$ is the automorphism group of a cograph, by extending the example in~\cite[Figure 3]{siblings}.} On the other hand, it is the automorphism group of some structure, and so this will show that cographs are not on top for infinitary bi-interpretability.
    
    Suppose that $\G$ is a cograph and let $a,b,c$ be three of its vertices. Suppose that $f$ is an automorphism of order $3$ such that $f(a)=b$, $f(b)=c$, and $f(c)=a$. We reason within the decomposition tree $\T$ of $\G$, recalling that $\G$ corresponds to the leaves of $\T$. Since $\G$ and $\T$ are infinitary bi-interpretable (\Cref{theorem:cotreescographsBBF}), they have the same automorphism group. An automorphism of $\G$ induces an automorphism of $\T$ by mapping robust modules of $\G$ to other robust modules.
    
     Hence $f$ acts on $\T$ as follows:
\[f(a \land b)=f(a)\land f(b)=b \land c, \ f(b \land c)=c \land a, \ f(a \land c)=a \land b\]
and
\[ f(a\land b\land c)=a\land b\land c.\]
At least two of $a \wedge b$, $b \wedge c$, and $c \wedge a$ are equal to $a \wedge b \wedge c$. This is because, e.g., if $a \wedge b \preceq a \wedge c$ then $a \wedge b = a \wedge b \wedge c$. Since this is fixed by $f$, and $f$ maps $a \wedge b \mapsto b \wedge c \mapsto c \wedge a \mapsto a \wedge b$, we must have that $a \wedge b = b \wedge c = c \wedge a = a \wedge b \wedge c$.
\[\Tree [. [.{$a \wedge b \wedge c$} [.$\vdots$ [.$a$ ] ] [.$\vdots$ [.$b$ ] ] [.$\vdots$ [.$c$ ]]] [.$\cdots$ ] ]\]

Now consider the following sets:
\begin{align*}
    A &= \{ u \in G : u \succeq a \wedge b \wedge c \text{ and } u \wedge a \succ a \wedge b \wedge c\} \\
    B &= \{ u \in G : u \succeq a \wedge b \wedge c \text{ and } u \wedge b \succ a \wedge b \wedge c\} \\
    C &= \{ u \in G : u \succeq a \wedge b \wedge c \text{ and } u \wedge c \succ a \wedge b \wedge c\} \\
    D &= G - A - B - C
\end{align*}
The map $f$ maps, setwise, $A$ to $B$, $B$ to $C$, and $C$ to $A$, and fixes $D$ setwise. Then define
\[ g(u)=\begin{cases}
        f(u) & u \in A\\
        f^{-1}(u)& u \in B\\
        u & u \in C \cup D
\end{cases}.\]
This is an automorphism of $\G$ which is its own inverse. Thus if $\G$ has an automorphism of order three, it must have an automorphism of order two, and so cannot have automorphism group $\mathbb{Z}_3$.
\end{proof}

The following proof construction relies on ideas from~\cite{booth1979problems} proving that $\finfree{\path{4}}$ is not $GI$-complete, i.e., not every graph isomorphism problem polynomial-time reduces to it.

\begin{theorem}
\label{theorem:universalcondition}
For a finite graph $\G$, the following are equivalent:
    \begin{enumerate}
        \item[(i)] $\free{\G}$ is on top for effective and infinitary bi-interpretability;
        \item[(ii)] $\free{\G} \not\subseteq \free{\path{4}}$.
    \end{enumerate}
\end{theorem}

\begin{proof}
To show that (i) implies (ii) notice that, we prove the contrapositive. By \Cref{theorem:cotreesnotuniversal}, $\free{\path{4}}$ is not on top for infinitary bi-interpretability. For any class $\mathcal C\subseteq \free{\G}$, if $\mathcal C$ was on top for infinitary bi-interpretability, then $\free{\G}$ would be as well. Thus, in particular, $\free{\G}$ is not on top for infinitary bi-interpretability if $\free{\G}\subseteq \free{\path{4}}$. If such $\free{\G}$ was on top for effective bi-interpretability, then it would be on top for infinitary bi-interpretability and hence, it cannot be on top for effective bi-interpretability.

To see that (ii) implies (i), suppose that $\G$ is not an induced subgraph of $\path{4}$. To establish that $\free{\G}$ is on top for effective bi-interpretability---and thus also for infinitary bi-interpretability---it suffices to consider the case where $\G$ contains a cycle as an induced subgraph. This is 
justified by the following two facts:
\begin{enumerate}[label={\alph*)}]
    \item a graph $\G$ is an induced subgraph of $\path{4}$ if and only if $\G$ and $\bar{\G}$ do not have a cycle as an induced subgraph. This fact is folklore mentioned in~\cite[Section 4.7]{booth1979problems}. The left-to-right direction is immediate combining the facts that $\overline{\path{4}}=\path{4}$ and that any induced subgraph of $\path{4}$ cannot contain a cycle. The opposite direction is immediate for graphs of at least $6$ vertices (this follows from the fact that the $R(3,3)=6)$), and it can be proved for vertices of $5$ or less vertices exhausting all the possible cases.
    \label{item:p4-cycle}
    \item The classes $\free{\G}$ and $\free{\overline{\G}}$ are trivially CBF-reducible to each other.
\end{enumerate}

Since the class of graphs is on top for effective bi-interpretability~\cite{computablefunctors}, and CBF-reducibility is transitive, it suffices to show that the class of graphs is CBF-reducible to $\free{\G}$. We will define a computable functor $F=(\Phi,\Phi_*): Graphs \to \free{\G}$. Let $\hat{\D}$ be the isomorphism
closure of $F(Graphs)$. We will then define a computable functor $G=(\Psi,\Psi_*):\hat{\D}\to Graphs$ satisfying the conditions in \cref{def:cbf}. Given a graph $\G$ that contains a cycle, let
$n=1+\max\{k:\cycle{k}\text{ is an induced subgraph of }\G\}$ and let
$q=|\G|+1$. For every $m\geq 4$, let $\mathcal{R}_m$ be the graph
obtained from $\cycle{m}$ by choosing two adjacent vertices
$r_m,s_m$ and attaching a new leaf to $s_m$. We call $r_m$ the
\define{attachment vertex} of $\mathcal{R}_m$.

Notice that $\mathcal{R}_m$ is rigid once its attachment vertex is
fixed and that the only cycle in $\mathcal{R}_m$ is the original copy of $\cycle{m}$.

\paragraph{Definition of $\Phi$.} Given a graph $\H$, let $\Phi(\H)$ be the graph obtained as follows. For an example see~\cref{fig:defphi}.
\begin{enumerate}
   \item For every $v \in H$, add to $\Phi(\H)$ a point
  $\Phi(\H)(v)$ and make it the attachment vertex of a copy of $\mathcal{R}_n$.
 \item Then, for every pair of distinct vertices $v,w\in H$, if
  $v\E w$, connect them via a path of length $q$ and make each
  vertex in this path, other than the endpoints, the attachment
  vertex of a copy of $\mathcal{R}_{n+1}$. If $v\notE w$, do the same using
  copies of $\mathcal{R}_{n+2}$.
\end{enumerate}

\begin{figure}[H]
   \centering\includegraphics{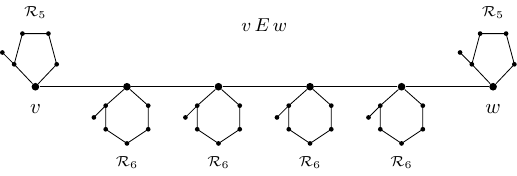}
   \caption{\label{fig:defphi}
An example of the operator $\Phi$ for two adjacent vertices $v,w$,
in the case $n=5$ and $q=6$. The endpoints are attachment vertices
of copies of $\mathcal{R}_5$, while the internal vertices of the
connecting copy of $\path{6}$ are attachment vertices of copies of
$\mathcal{R}_6$}
\end{figure}
We need to show that given a graph $\H$, $\Phi(\H) \in \free{\G}$ and to do so, since $\cycle{n-1}$ is an induced subgraph of $\G$, it suffices to show that $\cycle{n-1}$ is not an induced subgraph of $\Phi(\H)$.

We will refer to the copies of $\mathcal R_m$ introduced in the definition of $\Phi$ above as \define{marker gadgets}.
Since the only cycle in $\mathcal{R}_m$ is its copy of $\cycle{m}$, none of the marker gadgets contains a copy of $\cycle{n-1}$. Moreover, no copy of $\cycle{n-1}$ can contain a non-attachment vertex of a marker gadget, since each marker meets the rest of $\Phi(\H)$ only at its attachment vertex.
 It remains to exclude the case in which the copy of $\cycle{n-1}$ consists of vertices coming from the copies of $\path{q}$. Since each copy of $\path{q}$ has $q-1=|G|$ edges and, by construction, any cycle using only these paths must contain at least three such paths, every such cycle has length at least $3(|G|)>n+2$. In particular, every cycle in $\Phi(\H)$ of length $n$, $n+1$, or $n+2$ is contained in one of the corresponding marker gadgets.

By construction  $\H \cong \H'$ implies $\Phi(\H) \cong \Phi(\H')$: the converse implication will follow from the definition of $\Psi$ given below. Furthermore, as the definition of $\Phi(\H)$ is uniform and uses no information beyond that supplied by the atomic diagram of $\H$, $\Phi$ is a computable operator.

\paragraph{Definition of $\Phi_*$.}
Let $f$ be an isomorphism between two graphs $\H$ and $\H'$. We define the isomorphism $\Phi_*(f)$ from $\Phi(\H)$ to $\Phi(\H')$ as follows.
\begin{itemize}
\item if $v \in H$, $\Phi_*(f)$ maps $\Phi(\H)(v)$ to $\Phi(\H')(f(v))$. Furthermore, $\Phi_*(f)$ maps the copy of $\mathcal{R}_n$ attached to $\Phi(\H)(v)$ to the one attached to $\Phi(\H')(f(v))$;
    \item for every pair of distinct vertices $v,w\in H$, consider $\Phi(\H)(v)$ and $\Phi(\H)(w)$. By definition of $\Phi$, in
$\Phi(\H)$ there is a copy of $\path{q}$ with endpoints
$\Phi(\H)(v)$ and $\Phi(\H)(w)$ and in which every vertex that is
not $\Phi(\H)(v)$ or $\Phi(\H)(w)$ has attached a copy of
$\mathcal{R}_{n+1}$ or $\mathcal{R}_{n+2}$, depending on whether
$(v,w)$ is an edge in $\H$ or not. The isomorphism $\Phi_*(f)$
maps this induced subgraph to the one having endpoints
$\Phi(\H')(f(v))$ and $\Phi(\H')(f(w))$.
\end{itemize}
Since each $\mathcal{R}_m$ is rigid once its attachment vertex is
fixed, once the image of its attachment vertex is specified there is
a unique isomorphism between the corresponding marker gadgets.
Likewise, once the images of the two endpoints of one of the
connecting paths are specified, the isomorphism of that path is
unique. Hence $\Phi_*(f)$ is uniquely determined. In particular,
$\Phi_*$ preserves identities and composition. Notice that $\Phi_*(f)$ is an isomorphism between $\Phi(\H)$ and $\Phi(\H')$ which is computable relative to  $\H$, $f$ and $\H'$.

\paragraph{Definition of $\Psi$.} Given $\H\in\hat{\D}$, call a vertex $v\in H$ designated if it has
at least five distinct neighbors. Since $\H$ is isomorphic to a
graph in $F(Graphs)$, these are exactly the attachment vertices of
the copies of $\mathcal{R}_n$, i.e., those vertices coming from vertices in the original graph (every such vertex has infinite
degree, while every other vertex has degree at most $4$). Since having at least five distinct neighbors is an existential
condition, we can computably enumerate the designated vertices
without repetition, say $d_0^\H,d_1^\H,\ldots$.

On the other hand, a vertex is not designated if it is part of a copy of $\mathcal{R}_n$ and another vertex in this copy has degree greater than $4$ or if it is part of $\mathcal{R_m}$ with $m\in \{n+1,n+2\}$. This is again an existential condition and thus the set of designated vertices is computable. 

Let the universe of $\Psi(\H)$ be $\omega$ and for every $i\neq j$, search in parallel for a copy of $\path{q}$
with endpoints $d_i^\H,d_j^\H$ whose internal vertices are
attachment vertices of copies of $\mathcal{R}_{n+1}$, and for a
copy of $\path{q}$ whose internal vertices are attachment vertices
of copies of $\mathcal{R}_{n+2}$. Exactly one of these searches
succeeds. We let $(i,j)$ be an edge in $\Psi(\H)$ if and only if
the first search succeeds. This gives the atomic diagram of $\Psi(\H)$. It is clearly computable uniformly from the atomic diagram of $\H$.

\paragraph{Definition of $\Psi_*$}. Let $f$ be an isomorphism between two graphs $\H, \H' \in \hat{\D}\subseteq \free{\G}$. We define the isomorphism $\Psi_*(f)$ from $\Psi(\H)$ to $\Psi(\H')$ by letting $\Psi_*(f)(i)=j$ if and only if $f(d_i^\H)=d_j^{\H'}$. Since isomorphisms preserve designated vertices, such a $j$ exists
uniquely and can be found computably. Thus $\Psi_*(f)$ is an
isomorphism computable relative to $\H$, $f$, and $\H'$. It is
immediate from the definition that $\Psi_*$ preserves identities
and composition.

\paragraph{$F$ and $G$ are pseudo-inverses.} It remains to show that $F=(\Phi,\Phi_*)$ and $G=(\Psi,\Psi_*)$ are effective pseudo-inverses.
First, let $\H$ be a graph. For each $i$, let $v_i$ be the unique vertex of $\H$ such that $d_i^{F(\H)}=\Phi(\H)(v_i)$, and define $\Lambda^\H(i)=v_i$. It is easy to see that $\Lambda^\H:G(F(\H))\to\H$ is computable. Moreover, if $f:\H_1\cong\H_2$, then $\Lambda^{\H_2}\circ G(F(f))=f\circ\Lambda^{\H_1}$, since both sides send $i$ to the image under $f$ of the vertex corresponding to $d_i^{F(\H_1)}$. Thus, $\Lambda$ is an effective isomorphism between $GF$ and $id$.

Now let $\H\in\hat{\D}$. The graph $F(G(\H))$ reconstructs the designated vertices, connecting paths, and marker gadgets of $\H$. Define $\Gamma^\H:F(G(\H))\to\H$ by sending the designated vertex corresponding to $i$ to $d_i^\H$. For each pair $i,j$, map the corresponding connecting copy of $\path q$ in $F(G(\H))$ to the connecting copy of $\path q$ between $d_i^\H$ and $d_j^\H$ identified in the computation of $\Psi(\H)$.

The marker gadgets are recovered uniformly as well. Once the attachment vertex of a copy of $\mathcal R_m$ is known, search for the corresponding copy of $\cycle m$. Its vertex $s_m$ is the unique cycle-neighbor of the attachment vertex having a neighbor outside the cycle; this also identifies the attached leaf, and the rest of the cycle is then determined. Hence $\Gamma^\H$ is computable uniformly from $\H$.

If $f:\H\cong\H'$ is an isomorphism, then $f\circ\Gamma^\H=\Gamma^{\H'}\circ F(G(f))$ and thus $\Gamma$ is an effective isomorphism between $FG$ and $id$. Indeed, the two maps agree on designated vertices, and therefore also on the connecting paths and marker gadgets determined by them.

Finally, the two families of isomorphisms are compatible with $F$ and $G$. It is routine to check that for every graph $\H$, $F(\Lambda^\H)=\Gamma^{F(\H)}$, and for every $\H\in\hat{\D}$, $G(\Gamma^\H)=\Lambda^{G(\H)}$. Thus, $F$ and $G$ satisfy the requirements for effective pseudo-inverses.
\end{proof}

The following theorem is the analogue of \Cref{theorem:universalcondition} for Borel reducibility. Notice that in \Cref{theorem:universalcondition}
we did not have to consider the case $\free{\G}\subseteq \free{\overline{\path{4}}}$ as $\path{4}=\overline{\path{4}}$ (this fact was used in the direction (ii) implies (i) of the proof). 

\begin{theorem}
\label{theorem:characterizationborelcomplete}
For a finite graph $\G$, the following are equivalent:
    \begin{enumerate}
        \item[(i)] $\free{\G}$ is on top for Borel reducibility;
        \item[(ii)] $\free{\G}\not\subseteq\free{\path{3}}$ and $\free{\G}\not\subseteq \free{\overline{\path{3}}}$.
    \end{enumerate}
\end{theorem}
\begin{proof}
To prove (i) implies (ii) we prove the contrapositive. To do so it suffices to combine the following two facts: (a)
both $\free{\path{3}}$ and $\free{\overline{\path{3}}}$ are not on top for Borel reducibility, and (b) the property of not being on top for Borel reducibility is closed under subset. Property (b) just follows by definition of Borel reducibility. Property (a) follows from the following two facts: 
\begin{itemize}
\item a graph in $\free{\path{3}}$ is just an equivalence structure, and it is well-known that deciding whether two equivalence structures are isomorphic is Borel (in particular it is $\Pi_4^0$). Indeed, to determine whether two equivalence structures are isomorphic, one has to compare the numbers of equivalence classes of each size (see e.g.~\cite[page 92]{within});
\item $\free{\overline{\path{3}}}$ Borel reduces to $\free{\path{3}}$ via the (Borel) function that, given in input a graph outputs its complement.

\end{itemize}

For the opposite direction, by \Cref{theorem:universalcondition}, we know that if $\free{\G} \not\subseteq \free{\path{4}}$, then $\free{\G}$  is on top for effective bi-interpretability and hence on top for Borel reducibility. \Cref{theorem:cographborelcomplete} tells us that $\free{\path{4}}$ is Borel complete. Notice that every graph $\G$ such that $\G$ is a proper induced subgraph of $\path{4}$ is either an induced subgraph of $\path{3}$ or $\overline{\path{3}}$. Hence, if $\free{\G}\subseteq\free{\path{4}}$, then by \Cref{lemma:forbiddenGH} and our assumptions, $\G$ cannot be a proper induced subgraph of $\path{4}$, and therefore $\G=\path{4}$. Thus, by \Cref{theorem:cographborelcomplete}, $\free{\G}$ is on top for Borel reducibility.
\end{proof}

As a corollary of Theorems~\ref{theorem:characterizationborelcomplete} and~\ref{theorem:universalcondition} and Lemma~\ref{lemma:forbiddenGH}, we get that $\free{\path{4}}$ is the only $\free\G$ class that is on top for Borel reducibility but not for effective bi-interpretability. In particular, we get the following as announced in the introduction.
\begin{theorem} \label{theorem:p4borelnotbiinterpret}
    $\free{\path{4}}$ is on top for Borel reducibility but not for effective bi-interpretability.
\end{theorem}

\printbibliography
\end{document}